\documentclass[pdflatex]{sn-jnl}

\usepackage{graphicx}%
\usepackage{amsmath,amssymb,amsfonts}%
\usepackage{amsthm}%
\usepackage{mathrsfs}%
\usepackage{xcolor}%
\usepackage{anyfontsize}%
\usepackage{enumitem}%

\newtheorem{theorem}{Theorem}%
\newtheorem{proposition}[theorem]{Proposition}%
\newtheorem{example}[theorem]{Example}%
\newtheorem{remark}[theorem]{Remark}%

\begin{document}

\title{On the restriction to unitarity for rational approximations to the exponential function}
\author{Tobias Jawecki\footnote{Email: tobias.jawecki@gmail.com}}

\abstract{In the present work we consider rational best approximations to the exponential function that minimize a uniform error on a subset of the imaginary axis. Namely, Chebyshev approximation and unitary best approximation where the latter is subject to further restriction to unitarity, i.e., requiring that the imaginary axis is mapped to the unit circle. We show that Chebyshev approximants are not unitary, and consequently, distinct to unitary best approximants. However, unitary best approximation attains at most twice the error of Chebyshev approximation, and thus, the restriction to unitarity is not a severe restriction in a practical setting. Moreover, Chebyshev approximation and unitary best approximation attain the same asymptotic error as the underlying domain of approximation shrinks to the origin.}
\maketitle

\section{Introduction}

We consider rational approximations to the exponential function on a subset of the imaginary axis, namely,
\begin{equation}\label{eq:defapprox}
r(\mathrm{i} x) = p(\mathrm{i} x)/q(\mathrm{i} x) \approx \mathrm{e}^{\mathrm{i} \omega x},~~~\omega \geq 0,~~x\in[-1,1],
\end{equation}
where $p$ and $q$ are polynomials of degree $\leq n$, and $\omega$ is also referred to as {\em frequency}.
While $r$ in~\eqref{eq:defapprox} might be most accurately referred to as an approximation to $\mathrm{e}^{\omega z}$ for $z\in\mathrm{i} [-1,1]$, we abuse notation slightly and refer to $r$ as an approximation to $\mathrm{e}^{\mathrm{i}\omega x}$ and $\mathrm{e}^{\omega z}$ interchangeably throughout the present work, as convenient in a particular context.
The focus of the present work is on rational approximations which minimize the uniform error
\begin{equation}\label{eq:uniformerror}
\max_{x\in[-1,1]} |r(\mathrm{i} x) - \mathrm{e}^{\mathrm{i} \omega x}|.
\end{equation}
To this end, for the uniform norm of a complex function $g$ we use the notation
\begin{equation*}
\| g \| := \max_{x\in[-1,1]} |g(\mathrm{i} x)|,
\end{equation*}
and for the set of rational functions with polynomials of degree $\leq n$ in the numerator and denominator we use the notation
\begin{equation*}
\mathcal{R}_n := \{p/q~|~\text{where $p$ and $q$ are polynomials of degree $\leq n$ and $q\not\equiv 0$}\}.
\end{equation*}
A rational function $ r^c\in\mathcal{R}_n$ which minimizes the uniform error~\eqref{eq:uniformerror} is referred to as {\em Chebyshev approximant}, i.e.,
\begin{equation}\label{eq:bestrat}
\| r^c - \exp(\omega \cdot)\|
= \min_{r\in\mathcal{R}_n} \| r - \exp(\omega \cdot)\| ~~~ =: E_{n,\omega}^c.
\end{equation}
In the present setting Chebyshev approximants exist~\cite[Theorem~III]{Wa31}, but might not be unique as in general uniqueness of complex Chebyshev approximants is not always given~\cite{SV78,GT83,Gu83}. To our knowledge more specific results for rational Chebyshev approximation to $\mathrm{e}^{\mathrm{i} \omega x}$ are not available in literature. Thus, for the characterization of Chebyshev approximants~\eqref{eq:bestrat} we may only refer to general results on complex rational Chebyshev approximation, for instance~\cite{Wi79,Wu80,Gu83,Ru85,TI93} and related works with a focus on algorithms~\cite{EW76,IT93}. We remark that equioscillating error curves are a typical property of Chebyshev approximation in a real setting, see~\cite[Chapter~23]{Tre13} for instance, but similar results do not hold true for complex problems such as~\eqref{eq:bestrat} in general.

\smallskip
Besides Chebyshev approximation we also consider unitary best approximation to $\mathrm{e}^{\mathrm{i} \omega x}$ which was recently introduced in \cite{JS23u}.
A rational function $r\in\mathcal{R}_n$ is referred to as unitary if
\begin{equation*}
|r(\mathrm{i} x)| = 1,~~~x\in\mathbb{R},
\end{equation*}
and the set of unitary rational functions for a degree $n$ is denoted as
\begin{equation*}
\mathcal{U}_n := \{r\in\mathcal{R}_n~|~\text{$|r(\mathrm{i} x)|=1$ for $x\in\mathbb{R}$}\}\subset\mathcal{R}_n.
\end{equation*}
A unitary rational function $r^u\in\mathcal{U}_n$ which minimizes the uniform error~\eqref{eq:uniformerror} in $\mathcal{U}_n$ is referred to as {\em unitary best approximant}, i.e.,
\begin{equation*}
\| r^u - \exp(\omega \cdot)\|
= \min_{r\in\mathcal{U}_n} \| r - \exp(\omega \cdot)\|
~~~ =: E_{n,\omega}^u.
\end{equation*}
For existence of unitary best approximants we refer to~\cite[Section~3]{JS23u}. Moreover, following~\cite[Section~5]{JS23u} the unitary best approximation is uniquely characterized by an equioscillating {\em phase error} for $\omega\in(0,(n+1)\pi)$. Otherwise, for $\omega=0$ the approximation $r^u\equiv 1$ is exact, and for $\omega\geq (n+1)\pi$ all $r\in\mathcal{U}_n$ attain the minimal error $E^u_{n,\omega}=2$~\cite[Proposition~4.5]{JS23u}.

Moreover, the unitary best approximant can be computed with sufficient accuracy for practical choice of $n$ and $\omega$ which makes this approach practical. For this purpose the use of the AAA--Lawson method~\cite{NST18,NT20} or interpolation-based algorithms motivated by the brasil algorithm~\cite{Ho21} are suggested in~\cite{JS23u}. Unitarity of approximations generated by these algorithms goes back to~\cite{JS24}.

\smallskip
Rational approximation to the exponential function has some relevance for the matrix exponential and time integration of differential equations~\cite{ML03}. In particular, approximations to $\mathrm{e}^{\mathrm{i} \omega x}$ are related to numerical integration of skew-Hermitian systems and occur in equations of quantum mechanics such as the Schr{\"o}dinger equation~\cite{Lu08}. In this context, unitary and symmetric approximations have some relevance for geometric numerical integration~\cite{HLW06}. Following~\cite{JS23u}, the unitary best approximation satisfies these properties, as well as stability properties related to~\cite{HW02}. Thus, the restriction to unitarity has some benefits when considering best approximations for applications related to differential equations.

\smallskip
While the unitary best approximation satisfies some desirable properties, the Chebyshev approximation is potentially more accurate since $\mathcal{R}_n\supset \mathcal{U}_n$ implies $ E_{n,\omega}^c \leq E_{n,\omega}^u$. It hasn't been previously explored if the restriction to unitarity has a critical impact on accuracy, or whether Chebyshev approximants could be unitary and coincide with the unitary best approximant.

In the present work, we show that the unitary best approximation attains at most twice the error of the Chebyshev approximation, and Chebyshev approximants are not unitary, namely, $E_{n,\omega}^u/2 \leq E_{n,\omega}^c < E_{n,\omega}^u$ for $\omega>0$ (Theorem~\ref{thm:errinequals}). Moreover, for $\omega\to 0$ the asymptotic errors of the unitary best approximation and Chebyshev approximation coincide (Proposition~\ref{prop:asymerr}).

\section{Main results}


\begin{theorem}\label{thm:errinequals}
Provided $\omega >0$, the uniform errors $E_{n,\omega}^u$ and $E_{n,\omega}^c$ of the unitary best approximation and the Chebyshev approximation, respectively, satisfy
\begin{equation}\label{eq:inequalityerrors}
E_{n,\omega}^u/2 \leq E_{n,\omega}^c < E_{n,\omega}^u.
\end{equation}
This includes the case $\omega\geq (n+1)\pi$ for which the Chebyshev approximant corresponds to $r^c \equiv 0$ with $E_{n,\omega}^c=1$, and all $r^u\in\mathcal{U}_n$ are unitary best approximants with $E_{n,\omega}^u=2$. 
Moreover, for the trivial case $\omega=0$ the best approximants $r^c\equiv r^u\equiv 1$ are exact, i.e., $E_{n,\omega}^u= E_{n,\omega}^c=0$.
\end{theorem}
\begin{proof}
The proof of this theorem is topic of Section~\ref{sec:errorineq}.
\end{proof}


While Theorem~\ref{thm:errinequals} shows that the errors of the unitary best approximation and Chebyshev approximations are distinct for $\omega>0$, the following proposition shows that these approximations attain the same asymptotic error in the limit $\omega\to 0$.
\begin{proposition}\label{prop:asymerr}
The unitary best approximation and the Chebyshev approximation attain the asymptotic error
\begin{equation}\label{eq:asymerr}
E_{n,\omega}^u,~E_{n,\omega}^c
= \frac{2^{-2n} (n!)^2}{(2n)!(2n+1)!} \omega^{2n+1}
+ \mathcal{O}(\omega^{2n+2}),~~~\omega\to 0.
\end{equation}
\end{proposition}
Before providing the proof of Proposition~\ref{prop:asymerr}, we state the following auxiliary remark.
\begin{remark}\label{rmk:approxeixonscaleddomain}
Chebyshev approximation to $\mathrm{e}^{\mathrm{i} \omega x}$ for $x\in[-1,1]$ can be understood as Chebyshev approximation to $\mathrm{e}^z$ for $z\in\mathrm{i} [-\omega,\omega]$, and we have the identity
\begin{equation}\label{eq:EcasKeps}
E_{n,\omega}^c = \min_{r\in\mathcal{R}_n} \max_{z\in\mathrm{i} [-\omega,\omega]} | r(z) - \mathrm{e}^z|.
\end{equation}
In particular, this implies that $E_{n,\omega}^c$ is monotonically increasing in $\omega$, i.e., for $\omega_1\leq\omega_2$ the interval $\mathrm{i} [-\omega_1,\omega_1]$ is a subset of $\mathrm{i} [-\omega_2,\omega_2]$ which implies $E_{n,\omega_1}^c\leq E_{n,\omega_2}^c$.
\end{remark}

\begin{proof}[Proof of Proposition~\ref{prop:asymerr}]
The representation~\eqref{eq:EcasKeps} for $E_{n,\omega}^c$ fits to the setting of~\cite{Ja24cheb}, namely, as the error of a rational Chebyshev approximant to $f(z)=\mathrm{e}^{z}$ on the scaled domain $\varepsilon K$ where $K=\mathrm{i}[-1,1]$ and $\varepsilon =\omega$.
Thus, for $\omega\to 0$ the error representation in \cite[Theorem~14.(iii)]{Ja24cheb}, substituting $m=n$, $a_{nn}=(n!)^2/((2n)!(2n+1)!)$ and $t_{2n}=2^{-2n}$ therein, shows~\eqref{eq:asymerr} for $E_{n,\omega}^c$. On the other hand, for the error $E_{n,\omega}^u$ the asymptotic representation~\eqref{eq:asymerr} follows~\cite[Proposition~8.3]{JS23u}. 
\end{proof}

\begin{example}
Consider the case $n=0$ and $\omega>0$. For $n=0$ the approximants $r^c$ and $r^u$ correspond to constant numbers. Moreover, $r^c$ complies with a polynomial approximant of degree zero, and following \cite[p.~27]{Me67}, polynomial Chebyshev approximants to $\mathrm{e}^{\mathrm{i} \omega x}$ are symmetric. This carries over to the rational Chebyshev approximant in $\mathcal{R}_n$ for $n=0$, i.e., $\overline{r^c(\mathrm{i} x)}=r^c(-\mathrm{i} x)$, which implies that $r^c$ is a real-valued constant number.

We proceed to specify $r^c$ by first considering the case $\omega\in(0,\pi/2]$. For $r^c\in\mathbb{R}$ the point-wise error satisfies
\begin{equation}\label{eq:Echebn0lowerbound}
|r^c- \mathrm{e}^{\mathrm{i} \omega x}|= \left((\cos(\omega x) - r^c)^2+(\sin(\omega x))^2\right)^{1/2} \geq |\sin(\omega x)|.
\end{equation}
Moreover, for $r^c \equiv \cos \omega$ the point-wise error yields
\begin{equation}\label{eq:Echebn0lowerbound0}
|\cos\omega- \mathrm{e}^{\mathrm{i} \omega x}|
= \left((\cos(\omega x) - \cos \omega)^2+(\sin(\omega x))^2\right)^{1/2},
\end{equation}
which we can further simplify by expanding the quadratic form therein,
\begin{equation}\label{eq:Echebn0lowerbound1}
(\cos(\omega x) - \cos \omega)^2+(\sin(\omega x))^2 
= 1 + (\cos\omega)^2-2\cos\omega\cos(\omega x).
\end{equation}
Since $\cos(\omega x)\geq\cos\omega\geq 0$ for $\omega\in(0,\pi/2]$ and $x\in[-1,1]$, the right-hand side of this identity is bounded by
\begin{equation}\label{eq:Echebn0lowerbound2}
1 + (\cos\omega)^2-2\cos\omega\cos(\omega x) \leq 1 - (\cos\omega)^2
= (\sin\omega)^2.
\end{equation}
Combining~\eqref{eq:Echebn0lowerbound0},~\eqref{eq:Echebn0lowerbound1} and~\eqref{eq:Echebn0lowerbound2} we observe
\begin{equation}\label{eq:Ecupperbound}
|\cos \omega - \mathrm{e}^{\mathrm{i} \omega x}|
\leq \sin \omega,~~~\omega\in(0,\pi/2],~~x\in[-1,1].
\end{equation}
The lower bound~\eqref{eq:Echebn0lowerbound} implies $E_{0,\omega}^c\geq \sin \omega$ since $|\sin(\omega x)|= \sin \omega$ for $x=\pm1$, and~\eqref{eq:Ecupperbound} shows that this lower bound for $E_{0,\omega}^c$ is attained by $r^c=\cos \omega$ for $\omega\in(0,\pi/2]$.
Since the error $E_{0,\omega}^c$ is monotonically increasing with $\omega$, see Remark~\ref{rmk:approxeixonscaleddomain}, we note that for $\omega>\pi/2$ the Chebyshev approximation attains the error $E_{0,\omega}^c\geq 1$. In particular, this error is attained by $r_0\equiv 0$, i.e., $E_{0,\omega}^c=1$ for $\omega\geq \pi/2$.
Thus, the Chebyshev approximant for $n=0$ corresponds to
\begin{equation*}
r^c(x) = 
\left\{\begin{array}{ll}
\cos\omega,~~~&\text{for $\omega\in(0,\pi/2]$, and}\\
0,~~~&\text{for $\omega>\pi/2$},
\end{array}\right.
\end{equation*}
and attains the error
\begin{equation}\label{eq:Ecn0asym}
E_{0,\omega}^c = 
\left\{\begin{array}{ll}
\sin\omega,~~~&\text{for $\omega\in(0,\pi/2]$, and}\\
1,~~~&\text{for $\omega>\pi/2$}.
\end{array}\right.
\end{equation}

We proceed to consider the unitary best approximation. Due to unitarity and symmetry properties, in particular, \cite[Corollary~6.3]{JS23u}, we have $r^u(0)=1$, and for $n=0$ this implies
\begin{equation*}
r^u\equiv1,~~~\text{for}~~\omega\in (0,\pi).
\end{equation*}
We note
\begin{equation*}
|1-\mathrm{e}^{\mathrm{i} \omega x}|
= |\mathrm{e}^{-\mathrm{i} \omega x/2}-\mathrm{e}^{\mathrm{i} \omega x/2}| = 2|\sin(\omega x/2)|,
\end{equation*}
which implies
\begin{equation}\label{eq:Eun0}
E^u_{0,\omega} = 2\sin(\omega/2),~~~\text{for}~~\omega\in(0,\pi).
\end{equation}
Together with~\eqref{eq:Ecn0asym} this implies that $E_{0,\omega}^u$ and $E_{0,\omega}^c$ satisfy
\begin{equation*}
E_{0,\omega}^u, E_{0,\omega}^c = \omega + \mathcal{O}(\omega^2),~~~\text{for}~~\omega\to0,
\end{equation*}
which verifies the result of Proposition~\ref{prop:asymerr} for $n=0$.
Moreover,
\begin{equation*}
E_{0,\omega}^u/2 = \sin(\omega/2) \leq \sin(\omega) = E_{0,\omega}^c,~~~\text{for}~~\omega\in(0,\pi/2),
\end{equation*}
and in a similar manner this holds true for $\omega>\pi/2$ with $E_{0,\omega}^c=1$, and $E^u_{n,\omega}$ as in~\eqref{eq:Eun0} for $\omega\in[\pi/2,\pi)$ and $E^u_{n,\omega}=2$ for $\omega>\pi$~\cite[Proposition~4.5]{JS23u}. Moreover, these identities satisfy $E_{0,\omega}^c<E_{0,\omega}^u$, which verifies the inequalities in Theorem~\ref{thm:errinequals} for the case $n=0$.
See \figurename~\ref{fig:exnzero}
for an illustration of $E_{0,\omega}^c$ and $E_{0,\omega}^u$ over $\omega\in(0,\pi)$.
\end{example}

\begin{figure}
\centering
\includegraphics{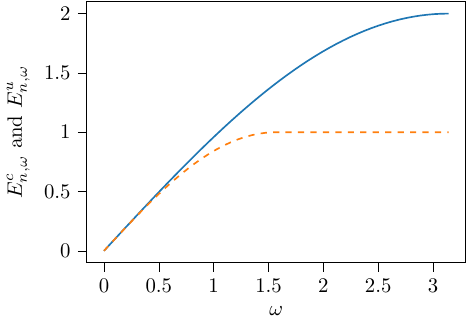}\hspace{1.5cm}
\caption{The errors of the unitary best approximation $E_{n,\omega}^u$ (solid line) and the Chebyshev approximation $E_{n,\omega}^c$ (dashed line)  for degree $n=0$ plotted over $\omega$.}
\label{fig:exnzero}
\end{figure}

\section{Proof of Theorem~\ref{thm:errinequals}}
\label{sec:errorineq}

In Proposition~\ref{prop:ruisnotrc} below, we show that the unitary best approximant is not a Chebyshev approximant for the case $\omega\in(0,(n+1)\pi)$ which is part of the proof of the upper bound in Theorem~\ref{thm:errinequals}. Proposition~\ref{prop:valleepoissontype} further below
provides an auxiliary result which is used to prove the lower bound in Theorem~\ref{thm:errinequals}. The proof of Theorem~\ref{thm:errinequals} is provided at the end of the present section.

\smallskip
The local Kolmogorov criterion~\cite{Gu83,Ru85}, cf.~\cite{IT93} for an overview, yields a necessary condition for local rational best approximants (i.e., approximants which minimize the respective uniform error w.r.t.\ local perturbations in $\mathcal{R}_n$), and consequently, also for the Chebyshev approximation~\eqref{eq:bestrat}. Based on this criterion, we show the following proposition.
\begin{proposition}\label{prop:ruisnotrc}
For $\omega\in(0,(n+1)\pi)$ the unitary best approximant is not a Chebyshev approximant.
\end{proposition}

Before proving this proposition, we recall some required properties of unitary best approximation from~\cite{JS23u}. We use the notation $p^\dag$ as defined in~\cite{JS23u}, i.e., for a polynomial $p(z) = \sum_{j=0}^n a_j z^j$, where $a_0,\ldots,a_n\in\mathbb{C}$ denote the coefficients of $p$, we let $p^\dag$ denote the polynomial $p^\dag(z) := \sum_{j=0}^n \overline{a}_j (-z)^j$. In particular, $p^\dag$ satisfies $p^\dag(\mathrm{i} x) = \overline{p(\mathrm{i} x)}$ for $x\in\mathbb{R}$.
Following~\cite[Proposition~2.1]{JS23u}, unitary rational functions $r\in\mathcal{U}_n$ are of the form $r=p^\dag/p$ where $p$ denotes a polynomial of degree $\leq n$, and this particularly holds true for the unitary best approximant, namely,
\begin{equation*}
r^u = \widetilde{p}^\dag/ \widetilde{p},~~~\text{for a polynomial $\widetilde{p}$ of degree exactly $n$,}
\end{equation*}
Moreover, $r^u$ is non-degenerate, i.e., $\widetilde{p}^\dag$ and $\widetilde{p}$ have no common zeros~\cite[Theorem~5.1.(ii)]{JS23u}.

For $\omega\in(0,(n+1)\pi)$ the unitary best approximant $r^u$ has the equivalent representation $r^u(\mathrm{i} x) = \mathrm{e}^{\mathrm{i} g(x)}$ where $g$ denotes a unique {\em phase function} as defined in~\cite[sections~4 and~5]{JS23u}. For $\omega\in(0,(n+1)\pi)$ the unitary best approximant $r^u(\mathrm{i} x) = \mathrm{e}^{\mathrm{i} g(x)} \approx \mathrm{e}^{\mathrm{i} \omega x}$ is uniquely characterized by an equioscillating phase error, i.e., there exist equioscillation points $\eta_1<\ldots<\eta_{2n+2}$ with $\eta_1=-1$ and $\eta_{n+2}=1$ for which the underlying phase function satisfies~\cite[eqs.~(5.1) and~(7.5)]{JS23u}, 
\begin{equation}\label{eq:equialpha}
g(\eta_j) -\omega \eta_j = (-1)^{j+1}\alpha,~~~j=1,\ldots,2n+2,
\end{equation}
where $ \alpha = \max_{x\in[-1,1]} |g(x)-\omega x|\in(0,\pi)$.
Following \cite[Corollary~5.2]{JS23u}, the unitary best approximant $r^u$ attains its maximum error exactly at the equioscillation points $\eta_1,\ldots,\eta_{2n+2}$, i.e.
\begin{equation}\label{eq:etamaxerr}
E_{n,\omega}^u = \| r^u - \exp(\omega\cdot)\|
= | r^u(\mathrm{i} \eta_j) - \mathrm{e}^{\mathrm{i} \omega \eta_j}|,~~~j=1,\ldots,2n+2.
\end{equation}

We proceed with the proof of Proposition~\ref{prop:ruisnotrc}.
\begin{proof}[Proof of Proposition~\ref{prop:ruisnotrc}]
Let $\omega\in(0,(n+1)\pi)$ be fixed, and let $r^u=\widetilde{p}^\dag/\widetilde{p}\in\mathcal{U}_n$ denote the unitary best approximant with $r^u(\mathrm{i} x) = \mathrm{e}^{\mathrm{i} g(x)}$ for the corresponding phase function~$g$.

For the remainder of the proof we may consider $r^u(\mathrm{i} x)$ as a function of $x\in[-1,1]$, and as an approximation to $f(x) = \mathrm{e}^{\mathrm{i} \omega x}$. In this setting we consider the local Kolmogorov criterion as in \cite[Theorem~1.2]{Ru85}. To apply this theorem we first recall that the unitary best approximant $r^u=\widetilde{p}^\dag/\widetilde{p}$ is non-degenerate, and thus, its numerator $\widetilde{p}^\dag$ and denominator $\widetilde{p}$ satisfy the condition to be relative prime. Moreover, the points of extreme error of $r^u$ correspond to the equioscillation points $\eta_1,\ldots,\eta_{2n+2}$. Applying \cite[Theorem~1.2]{Ru85}, we observe that $r^u$ is a local best approximation only if for all polynomials $\widehat{q}$ of degree $\leq 2n$ there exists at least one index $j\in\{1,\ldots,2n+2\}$ s.t.
\begin{equation}\label{eq:Regamgle0}
\operatorname{Re}(\overline{\gamma(\eta_j)} \widehat{q}(\eta_j)) \geq 0,
\end{equation}
where
\begin{equation}\label{eq:defgam}
\gamma(x) :=(\mathrm{e}^{\mathrm{i} \omega x} - r^u(\mathrm{i} x) ) \widetilde{p}(\mathrm{i} x)/\overline{\widetilde{p}(\mathrm{i} x)}.
\end{equation}
We proceed to show that~\eqref{eq:Regamgle0} is false for $\widehat{q}\equiv 1$, i.e.,
\begin{equation}\label{eq:Regamgle0g1}
\operatorname{Re}(\gamma(\eta_j)) < 0,~~~j=1,\ldots,2n+2.
\end{equation}
Due to $\overline{\widetilde{p}(\mathrm{i} x)} = \widetilde{p}^\dag(\mathrm{i} x)$ and $r^u=\widetilde{p}^\dag/\widetilde{p}$ we have $\widetilde{p}(\mathrm{i} x)/\overline{\widetilde{p}(\mathrm{i} x)}=r^u(\mathrm{i} x)^{-1}$ for $x\in\mathbb{R}$. Moreover, since $r^u$ is unitary, i.e., $|r(\mathrm{i} x)|=1$, we have $r^u(\mathrm{i} x)^{-1} = \overline{r^u(\mathrm{i} x)}$ for $x\in\mathbb{R}$. Thus,~\eqref{eq:defgam} simplifies to
\begin{equation}\label{eq:gamxsimple}
(\mathrm{e}^{\mathrm{i} \omega x} - r^u(\mathrm{i} x) ) \widetilde{p}(\mathrm{i} x)/\overline{\widetilde{p}(\mathrm{i} x)}
= (\mathrm{e}^{\mathrm{i} \omega x} - r^u(\mathrm{i} x))r^u(\mathrm{i} x)^{-1}
= (\mathrm{e}^{\mathrm{i} \omega x}\overline{r^u(\mathrm{i} x)} -1 ) .
\end{equation}

As a consequence of~\eqref{eq:equialpha}, at the equioscillation points $\eta_j$ the unitary best approximant $r^u(\mathrm{i} x)=\mathrm{e}^{\mathrm{i} g(x)}$ satisfies
\begin{equation}\label{eq:uateta}
\mathrm{e}^{\mathrm{i} \omega \eta_j}\overline{r^u(\mathrm{i} \eta_j)}
= \mathrm{e}^{\mathrm{i} (\omega \eta_j-g(\eta_j))}
=\mathrm{e}^{ (-1)^{j}\mathrm{i}\alpha},~~~j=1,\ldots,2n+2.
\end{equation}
Thus, simplifying $\gamma$~\eqref{eq:defgam} as in~\eqref{eq:gamxsimple} for $x=\eta_j$ and making use of~\eqref{eq:uateta}, we arrive at
\begin{equation*}
\gamma(\eta_j) = \mathrm{e}^{(-1)^{j}\mathrm{i} \alpha} - 1,~~~\text{with}~~\operatorname{Re}(\gamma(\eta_j)) = \cos(\alpha) - 1,~~~\text{for $j=1,\ldots,2n+2$}.
\end{equation*}
Since $\alpha\in(0,\pi)$ for $\omega\in(0,(n+1)\pi)$,
this shows~\eqref{eq:Regamgle0g1} which entails that~\eqref{eq:Regamgle0} his false for $\widehat{q}\equiv 1$.
Following \cite[Theorem~1.2]{Ru85}, the unitary best approximant $r^u$ is not a local best approximation in $\mathcal{R}_n$, and particularly, $r^u$ is not a Chebyshev approximant~\eqref{eq:bestrat}, which proves our assertion.
\end{proof}

We proceed with an auxiliary result which is used to prove the lower bound in Theorem~\ref{thm:errinequals}. This result is based on the interpolatory property of unitary best approximation which is described in~\cite[Corollary~5.2]{JS23u}. Namely, for $\omega\in(0,(n+1)\pi)$ the equioscillation property of the phase error~\eqref{eq:equialpha} yields that $g(x)-\omega x$ changes its sign between two neighbouring points $\eta_j$ and $\eta_{j+1}$ for $j=1,\ldots,2n+1$. Thus, there exist points $x_1,\ldots,x_{2n+1}$ with $x_j\in(\eta_j,\eta_{j+1})$ s.t.\ $g(x_j) = \omega x_j$ which implies
\begin{equation}\label{eq:xjinterpolate}
r^u(\mathrm{i} x_j) = \mathrm{e}^{\mathrm{i} \omega x_j},~~~j=1,\ldots,2n+1.
\end{equation}
In particular, the points $x_1,\ldots,x_{2n+1}$ can be understood as interpolation nodes of $r^u$.

For the following proposition we consider a slightly more general setting, which applies for the unitary best approximant for $\omega\in(0,(n+1)\pi)$, and also covers constant approximants $r^u\equiv \pm1$ in case of $\omega\geq (n+1)\pi$. In particular, we consider approximants $\zeta(x)\approx f(x)$ for $x\in[-1,1]$ where $\zeta\in\mathcal{R}_{n-d}$ denotes a rational function with defect $d\geq0$ and $f\colon[-1,1]\to \mathbb{C}$ is assumed to be continuous. For the present work the relevant case is $f(x) = \mathrm{e}^{\mathrm{i} \omega x}$, and either $d=0$ or $d=n$.
Assume there exist $2n-d+1$ interpolation nodes $x_1,\ldots,x_{2n-d+1}$, i.e.,
\begin{subequations}\label{eq:nonuniforminterpsetting}
\begin{equation}
\zeta(x_j) = f(x_j),~~~j=1,\ldots,2n-d+1.
\end{equation}
Moreover, let $\eta_1,\ldots,\eta_{2n-d+2}$ denote points of local extrema of $|\zeta-f|$, s.t.
\begin{equation}
-1\leq \eta_1<x_1<\eta_2<\ldots<x_{2n-d+1}<\eta_{2n-d+2}\leq 1,
\end{equation}
and 
\begin{equation}
\begin{aligned}
\max_{x\in[-1,x_1)} |\zeta(x)-f(x)|
&= |\zeta(\eta_1)-f(\eta_1)|,\\
\max_{x\in(x_{j-1},x_j)} |\zeta(x)-f(x)|
&= |\zeta(\eta_j)-f(\eta_j)|,~~\text{$j=2,\ldots,2n-d+1$, and}\\
\max_{x\in(x_{2n-d+1},1]} |\zeta(x)-f(x)|
&= |\zeta(\eta_{2n-d+2})-f(\eta_{2n-d+2})|.
\end{aligned}
\end{equation}
\end{subequations}
For the present work, only the uniform case $|\zeta(\eta_j)-f(\eta_j)|=E^u_{n,\omega}$ for $j=1,\ldots,2n-d+2$, is relevant. However, the setting~\eqref{eq:nonuniforminterpsetting}, covering non-uniform cases, has some use in practice, e.g., when a computed approximant satisfies an equioscillation property up to a tolerance.

The following proposition is motivated by results of~\cite{Le86} which show that real rational best approximation to real functions attain at most twice the approximation error of complex rational best approximation.
\begin{proposition}\label{prop:valleepoissontype}
Let $f\colon[-1,1]\to \mathbb{C}$ be a continuous function, and let $\zeta\in\mathcal{R}_{n-d}$ satisfy~\eqref{eq:nonuniforminterpsetting}.
Then, the error of a rational best approximant to $f$ is bounded from below by
\begin{equation}\label{eq:lowerboundonehalf}
\frac{1}{2}\left(\min_{j=1,\ldots,2n-d+2} |\zeta(\eta_j)-f(\eta_j)|\right)
\leq \min_{r\in\mathcal{R}_n} \max_{x\in[-1,1]} |r(x)-f(x)|
\end{equation}
\end{proposition}
\begin{proof}
To simplify the notation in the present proof, we also refer to rational functions $r\in\mathcal{R}_k$ as $(k,k)$-rational functions.
We proceed to prove~\eqref{eq:lowerboundonehalf} by contradiction. Namely, assume that there exists $r\in\mathcal{R}_n$ with
\begin{equation}\label{eq:lowerboundonehalfwrong}
\max_{x\in[-1,1]} |r(x)-f(x)| < \varepsilon/2,~~~\text{for $\varepsilon := \min_{j=1,\ldots,2n-d+2}  |\zeta(\eta_j)-f(\eta_j)|$}.
\end{equation}

Under this assumption, we proceed to show that there exist points in each of the intervals $(\eta_j,x_j)$ and $(x_j,\eta_{j+1})$, $j=1,\ldots,2n-d+1$, at which $|\zeta-r|$ attains the value $\varepsilon/2$. For the interpolation node $x_j$ we have $\zeta(x_j)=f(x_j)$ and~\eqref{eq:lowerboundonehalfwrong} implies
\begin{equation}\label{eq:rmuatxj}
|\zeta(x_j)-r(x_j)| < \varepsilon/2,~~~j=1,\ldots,2n-d+1.
\end{equation}
On the other hand, for the points $\eta_j$ we recall $|\zeta(\eta_j)-f(\eta_j)|\geq \varepsilon$ and $|r(\eta_j)-f(\eta_j)|< \varepsilon/2$. Using an inverse triangular inequality we observe
\begin{equation}\label{eq:rmuatetaj}
|\zeta(\eta_j)-r(\eta_j)|
\geq ||\zeta(\eta_j)-f(\eta_j)| - |r(\eta_j)-f(\eta_j)||
> \varepsilon/2,
\end{equation}
for $j=1,\ldots,2n-d+2$.
Combining~\eqref{eq:rmuatxj} and~\eqref{eq:rmuatetaj}, in combination with continuity arguments, we observe that there exist points $y_1,\ldots,y_{2(2n-d+1)}$ with $y_{2j-1}\in(\eta_j,x_j)$ and $y_{2j} \in(x_j,\eta_{j+1})$ for $j=1,\ldots,2n-d+1$ s.t.
\begin{equation}\label{eq:rmuzeros}
|\zeta(y_j) - r(y_j)| = \varepsilon/2,~~~ j=1,\ldots,2(2n-d+1).
\end{equation}
Since $\zeta\in\mathcal{R}_{n-d}$ and $r\in\mathcal{R}_n$, the difference $\zeta-r$ is a rational function of degree $(2n-d,2n-d)$, and for $x\in\mathbb{R}$,
$$
|\zeta(x)-r(x)|^2 = (\overline{\zeta(x)}-\overline{r(x)})(\zeta(x)-r(x)),
$$
is a rational function of degree $(2(2n-d),2(2n-d))$.
In particular, for $x\in\mathbb{R}$ this carries over to $|\zeta(x)-r(x)|^2-\varepsilon^2/4$. Since this rational function has $2(2n-d+1)$ zeros in $\mathbb{R}$, namely, $y_1,\ldots,y_{2(2n-d+1)}$ as in~\eqref{eq:rmuzeros} and $2(2n-d+1) > 2(2n-d)$, it is zero for all $x\in\mathbb{R}$ and we conclude
$$
|\zeta(x)-r(x)| = \frac{\varepsilon}{2},~~~x\in\mathbb{R}.
$$
However, this is a contradiction to~\eqref{eq:rmuatxj} and~\eqref{eq:rmuatetaj}. Thus, there exists no $r\in\mathcal{R}_n$ with $\max_{x\in[-1,1]}|r(x)-f(x)|<\varepsilon/2$ which proves~\eqref{eq:lowerboundonehalf}.
\end{proof}
Following~\cite[Remark~2]{Le86}, the inequality~\eqref{eq:lowerboundonehalf} might be strict under additional conditions, however, these conditions are not fully clear to the authors of the present work.

We proceed with the proof of Theorem~\ref{thm:errinequals}.
\begin{proof}[Proof of Theorem~\ref{thm:errinequals}]
We first show~\eqref{eq:inequalityerrors} for the case $\omega\in(0,(n+1)\pi)$. To show the upper bound therein, we note that the inequality $E^c_{n,\omega}\leq E^u_{n,\omega}$ directly follows from $\mathcal{U}_n\subset \mathcal{R}_n$. Following Proposition~\ref{prop:ruisnotrc}, the unitary best approximant is not a Chebyshev approximant in $\mathcal{R}_n$ for $\omega\in(0,(n+1)\pi)$, i.e., $E^c_{n,\omega}\neq E^u_{n,\omega}$, which implies that this inequality is strict.

We proceed to show the lower bound in~\eqref{eq:inequalityerrors} for the case $\omega\in(0,(n+1)\pi)$. Let $r^u(\mathrm{i} x) \approx \mathrm{e}^{\mathrm{i} \omega x}$ denote the unitary best approximant. We recall that for $\omega\in(0,(n+1)\pi)$ the unitary best approximant attains equioscillation points $\eta_1<\ldots<\eta_{2n+2}$ with $|r^u(\mathrm{i} \eta_j)-\mathrm{e}^{\mathrm{i} \omega \eta_j}| = E_{n,\omega}^u$ for $j=1,\ldots,2n+2$ as in~\eqref{eq:etamaxerr}, and interpolation nodes $x_1<\ldots<x_{2n+1}$ as in~\eqref{eq:xjinterpolate} with $x_j\in(\eta_j,\eta_{j+1})$ for $j=1,\ldots,2n+1$. Thus, we may apply Proposition~\ref{prop:valleepoissontype} for $f(x) = \mathrm{e}^{\mathrm{i} \omega x}$, $\zeta(x) = r^u(\mathrm{i} x)$, $d=0$, and
\begin{equation}\label{eq:conditionrutoProp6}
\min_{j=1,\ldots,2n-d+2} |\zeta(\eta_j)-f(\eta_j)| =E_{n,\omega}^u,~~~\text{and}~~
\min_{r\in\mathcal{R}_n} \max_{x\in[-1,1]} |r(x)-f(x)| = E_{n,\omega}^c.
\end{equation}
Consequently, the lower bound in~\eqref{eq:inequalityerrors} follows from~\eqref{eq:lowerboundonehalf} for $\omega\in(0,(n+1)\pi)$.

It remains to show~\eqref{eq:inequalityerrors} for the case $\omega\geq (n+1)\pi$. Following \cite[Proposition~4.5]{JS23u}, any $r^u\in\mathcal{U}_n$ is a unitary best approximation with $ E_{n,\omega}^u = 2$ in this case. We proceed to show $E_{n,\omega}^c \geq 1$ by applying Proposition~\ref{prop:valleepoissontype} for $r^u\equiv (-1)^n \in\mathcal{U}_n$ which can be understood as rational function in $\mathcal{R}_0$, i.e., we have the defect $d=n$. Define the points $ \eta_j = (2(j-1)-n-1)\pi/\omega\in[-1,1]$ for $j=1,\ldots,n+2$ and $ x_j = (2(j-1)-n)\pi/\omega\in(\eta_j,\eta_{j+1})$ for $j=1,\ldots,n+1$.
Since $\mathrm{e}^{\mathrm{i} \omega x_j} = (-1)^n$ for $j=1,\ldots,n+1$, the points $x_1,\ldots,x_{n+1}$ can be understood as interpolation nodes of $r^u$, i.e., $r^u = \mathrm{e}^{\mathrm{i} \omega x_j} $.
Moreover since $\mathrm{e}^{\mathrm{i} \omega \eta_j} = (-1)^{n+1}$ for $j=1,\ldots,n+2$, the approximant $r^u$ attains its maximal error at the points $\eta_1,\ldots,\eta_{n+2}$, i.e., $|r^u -  \mathrm{e}^{\mathrm{i} \omega \eta_j}| = 2$ for $j=1,\ldots,n+2$. Applying Proposition~\ref{prop:valleepoissontype} for $f(x) = \mathrm{e}^{\mathrm{i} \omega x}$, $\zeta(x) \equiv r^u$, $d=n$, and~\eqref{eq:conditionrutoProp6} with $E^u_{n,\omega}=2$, we conclude $E_{n,\omega}^c\geq E_{n,\omega}^u/2 = 1$. Since $r^c\equiv 0\in\mathcal{R}_n$ attains this lower bound we conclude $E_{n,\omega}^c= 1$, which proves our assertion for $\omega\geq (n+1)\pi$.

Moreover, for $\omega =0$ we have $\mathrm{e}^{\mathrm{i} \omega x}\equiv 1$, and thus, $r^u\equiv r^c\equiv 1\in\mathcal{U}_n$ is an exact approximation which shows $E_{n,\omega}^u= E_{n,\omega}^c=0$ in this case.
\end{proof}

\end{document}